\documentclass[12pt,reqno,a4paper]{amsart}
\usepackage{extsizes}
\usepackage{blindtext}
\usepackage{fullpage}
\usepackage{mathtools}
\usepackage{longtable}
\usepackage{amsmath,amssymb,amsthm}
\usepackage{amscd}
\usepackage{bm}
\usepackage{color}
\usepackage{mathrsfs}
\usepackage{hyperref}
\usepackage{dsfont}
\usepackage{enumerate}
\usepackage{epsfig}
\usepackage{float, graphicx}
\usepackage{booktabs}
\usepackage{latexsym, amsxtra}
\usepackage{mathrsfs}
\usepackage{multicol}
\usepackage[normalem]{ulem}
\usepackage{psfrag}
\usepackage[parfill]{parskip}
\usepackage{stmaryrd}
\usepackage{tikz}
\usepackage[T1]{fontenc}
\usepackage{url}
\usepackage{verbatim}
\usepackage{indentfirst}
\usepackage{tikz-cd}
\usepackage{svg}

\makeatletter
\def\thm@space@setup{%
	\thm@preskip=2ex \thm@postskip=1.5ex
}
\makeatother

\hypersetup{hidelinks}

\newtheorem{thm}{Theorem}[section]
\newtheorem{prop}[thm]{Proposition}
\newtheorem{lem}[thm]{Lemma}

\theoremstyle{definition}

\theoremstyle{remark}
\newtheorem{rmk}[thm]{Remark}

\newcommand{\CC}{\mathbb{C}}
\newcommand{\QQ}{\mathbb{Q}}
\newcommand{\RR}{\mathbb{R}}
\newcommand{\ZZ}{\mathbb{Z}}
\newcommand{\PP}{\mathbb{P}}
\newcommand{\HH}{\mathbb{H}}

\newcommand{\PU}{\mathrm{PU}}

\newcommand{\reg}{\mathrm{reg}}

\title{Commensurability Relations Between Deligne--Mostow--Thurston and Ghazouani--Pirio Monodromy Groups}
\author{Chenglong Yu, Yiming Zhong}
\date{}

\newcommand{\Addresses}{{
		\bigskip
		\footnotesize
		
		C.~Yu, \textsc{Center for Mathematics and Interdisciplinary Sciences, Fudan University and Shanghai Institute for Mathematics and Interdisciplinary Sciences (SIMIS), Shanghai, China}\par\nopagebreak
		\textit{E-mail address}: \texttt{yuchenglong@simis.cn}
		
		\medskip
		
		Y.~Zhong, \textsc{Institute of Mathematical Sciences, ShanghaiTech University, Shanghai, China}\par\nopagebreak
		\textit{E-mail address}: \texttt{zhongym1@shanghaitech.edu.cn}
}}

\begin{document}

\begin{abstract}
We classify commensurability relations between Deligne--Mostow--Thurston monodromy groups and the sixteen arithmetic Ghazouani--Pirio monodromy groups arising from moduli of flat cone metrics on the sphere and on the torus. In both settings, the monodromy group preserves a skew-Hermitian form coming from twisted homology. We use a degeneration method to compute their determinant classes in the genus one case. The defining CM fields and determinant classes of the skew-Hermitian forms determine the commensurability relations. In particular, each of the sixteen arithmetic Ghazouani--Pirio monodromy groups is commensurable with an arithmetic Deligne--Mostow--Thurston monodromy group.
\end{abstract}

\maketitle

\setcounter{tocdepth}{1}

\section{Introduction}
 
Deligne and Mostow constructed complex hyperbolic monodromy groups as the images of monodromy representations on twisted homology with coefficients in rank-one unitary local systems on punctured projective lines \cite{deligne1986monodromy,mostow1986generalized,mostow1988discontinuous}.
Thurston studied these groups through flat cone metrics on the sphere \cite{thurston1998shapes}, and Veech extended this viewpoint to flat cone metrics on arbitrary genus surfaces \cite{veech1993flat}. In genus one, Ghazouani and Pirio constructed complex hyperbolic ball quotients associated with algebraic leaves in the moduli of flat cone metrics on elliptic curves \cite{ghazouani2017moduli,ghazouani2017flat}.
 
The discreteness and arithmeticity of Deligne--Mostow--Thurston monodromy groups are completely determined by Deligne, Mostow, and Thurston \cite{deligne1986monodromy,mostow1986generalized,mostow1988discontinuous,thurston1998shapes}. 
Ghazouani and Pirio obtained sixteen arithmetic lattices as monodromy groups \cite[\S11.3, Table 1]{ghazouani2017moduli}, but a general discreteness criterion for arbitrary Ghazouani--Pirio monodromy groups remains unknown.

A natural problem is to classify these groups up to commensurability. 
We say that two subgroups $\Gamma,\Gamma'\subset \PU(1,n)$ are commensurable up to conjugacy if $\Gamma\cap g\Gamma'g^{-1}$ has finite index in both $\Gamma$ and $g\Gamma'g^{-1}$ for some $g\in\PU(1,n)$. 
On the Deligne--Mostow--Thurston side, the commensurability classes are completely classified by Sauter, Deligne--Mostow, Kappes--M\"oller, McMullen, and Yu--Zheng \cite{sauter1990isomorphisms,deligne1993commensurabilities,kappes2016lyapunov,McMullen2017Gauss-Bonnet,yu2024comm}. On the Ghazouani--Pirio side, Ghazouani and Pirio use double covers from elliptic curves to $\PP^1$ to relate certain Ghazouani--Pirio and Deligne--Mostow--Thurston monodromy groups in $\PU(1,2)$ \cite[\S4.2.5]{ghazouani2017flat}.

The goal of this paper is to classify the commensurability relations between the arithmetic Ghazouani--Pirio examples listed in \cite[\S11.3, Table 1]{ghazouani2017moduli} and the arithmetic Deligne--Mostow--Thurston examples. 
Commensurable lattices necessarily act on complex hyperbolic balls of the same dimension. For the arithmetic lattices considered here,  the CM field is also a commensurability invariant (see Proposition \ref{prop:parity-arithmetic-criterion}).
A Ghazouani--Pirio datum with $r$ punctures corresponds to a complex hyperbolic ball of dimension $r-1$, and a Deligne--Mostow--Thurston datum with $N$ marked points corresponds to one of dimension $N-3$. Thus commensurability can occur only when $r=N-2$.
For $d\ge 1$, let $\zeta_d:=e^{2\pi i/d}$. The arithmetic Ghazouani--Pirio examples considered below have CM field $\QQ(i)$ or $\QQ(\zeta_3)$ and are labelled by $\ell\in\{a,b,\dots,p\}$ in \cite[\S 11.3, Table 1]{ghazouani2017moduli}. 
The following theorem gives the complete list of these commensurability relations.

\begin{thm}
\label{thm:main}
Let $\Gamma_\ell$ be one of the sixteen arithmetic Ghazouani--Pirio monodromy groups listed in \cite[\S 11.3, Table 1]{ghazouani2017moduli}. Then the arithmetic Deligne--Mostow--Thurston groups commensurable with $\Gamma_\ell$ are exactly those listed in Table \ref{tab:arithmetic-gp-dm-classification}. Each row represents one commensurability class, with the Ghazouani--Pirio labels in the third column and the Deligne--Mostow--Thurston tuples in the fourth.

\begin{table}[htbp]
\scriptsize
\centering
\caption{Arithmetic Ghazouani--Pirio groups and commensurable Deligne--Mostow--Thurston tuples.}
\begin{tabular}{ccll}
\hline
Dim. & Field & GP labels & Deligne--Mostow--Thurston tuples \\
\hline
2 & $\QQ(i)$ & $b,d$ & $\frac{1}{4}(1,1,1,2,3)$;\quad $\frac{1}{4}(1,1,2,2,2)$ \\
\hline
2 & $\QQ(\zeta_3)$ & $a,c,e,f,g,h,i$ & $\frac{1}{3}(1,1,1,1,2)$;\quad $\frac{1}{6}(1,1,1,4,5)$;\quad $\frac{1}{6}(1,1,2,3,5)$;\quad $\frac{1}{6}(1,1,2,4,4)$; \\
  &                &                   & $\frac{1}{6}(1,1,3,3,4)$;\quad $\frac{1}{6}(1,2,2,2,5)$;\quad $\frac{1}{6}(1,2,2,3,4)$;\quad $\frac{1}{6}(1,2,3,3,3)$; \\
  &                &                   & $\frac{1}{6}(2,2,2,3,3)$ \\
\hline
3 & $\QQ(i)$ & $j$ & $\frac{1}{4}(1,1,1,1,1,3)$;\quad $\frac{1}{4}(1,1,1,1,2,2)$ \\
\hline
3 & $\QQ(\zeta_3)$ & $l,m$ & $\frac{1}{3}(1,1,1,1,1,1)$;\quad $\frac{1}{6}(1,1,1,1,4,4)$;\quad $\frac{1}{6}(1,1,1,2,2,5)$; \\
  &                &       & $\frac{1}{6}(1,1,2,2,2,4)$;\quad $\frac{1}{6}(1,1,2,2,3,3)$ \\
\hline
3 & $\QQ(\zeta_3)$ & $k$ & $\frac{1}{6}(1,1,1,1,3,5)$;\quad $\frac{1}{6}(1,1,1,2,3,4)$; \\
  &                &     & $\frac{1}{6}(1,1,1,3,3,3)$;\quad $\frac{1}{6}(1,2,2,2,2,3)$ \\
\hline
4 & $\QQ(\zeta_3)$ & $n,o$ & $\frac{1}{6}(1,1,1,1,1,2,5)$;\quad $\frac{1}{6}(1,1,1,1,1,3,4)$;\quad $\frac{1}{6}(1,1,1,1,2,2,4)$; \\
  &                &       & $\frac{1}{6}(1,1,1,1,2,3,3)$;\quad $\frac{1}{6}(1,1,1,2,2,2,3)$;\quad $\frac{1}{6}(1,1,2,2,2,2,2)$ \\
\hline
5 & $\QQ(\zeta_3)$ & $p$ & $\frac{1}{6}(1,1,1,1,1,1,1,5)$;\quad $\frac{1}{6}(1,1,1,1,1,1,2,4)$; \\
  &                &     & $\frac{1}{6}(1,1,1,1,1,1,3,3)$;\quad $\frac{1}{6}(1,1,1,1,2,2,2,2)$ \\
\hline
\end{tabular}
\label{tab:arithmetic-gp-dm-classification}
\end{table}
\end{thm}

\begin{rmk}
In particular, each of the sixteen arithmetic Ghazouani--Pirio groups lies in the same commensurability class as an arithmetic Deligne--Mostow--Thurston group.
\end{rmk}

\begin{rmk}
The Ghazouani--Pirio examples realize every arithmetic Deligne--Mostow--Thurston commensurability class having the same ball dimension and CM field as one of them, except the dimension-$5$ class over $\QQ(\zeta_3)$ represented by $\frac{1}{6}(1,1,1,1,1,2,2,3)$.
\end{rmk}

The Deligne--Mostow--Thurston and Ghazouani--Pirio constructions define two skew-Hermitian forms over certain CM fields on their twisted homology groups, respectively. These skew-Hermitian forms are invariant by the corresponding monodromy actions.
The corresponding CM fields determine the arithmetic commensurability classes in the even ball dimensions. In the odd ball dimensions, we must also compare the determinant classes of the skew-Hermitian forms. 
The genus one degeneration formula (Proposition \ref{prop:elliptic-merge-punctures}) gives an orthogonal decomposition of the Ghazouani--Pirio skew-Hermitian form when two punctures collide. 
Then we apply the determinant formula (Proposition \ref{prop:elliptic-determinant-reduction}) to compute the corresponding determiannt classes. We compare these determinant classes with the commensurability list in \cite[\S9.2]{yu2024comm}, and the commensurability relations follow from Proposition \ref{prop:parity-arithmetic-criterion}.

\subsection*{Acknowledgements}
The authors would like to thank Selim Ghazouani for helpful suggestions.
The first author is supported by the national key research and development program of China (No. 2022YFA1007100) and NSFC 12201337.

\section{Two Skew-Hermitian Spaces}
\label{sec:lattices}

In the Deligne--Mostow and Ghazouani--Pirio constructions, we start with the twisted homology groups of rank one unitary local systems on a punctured projective line and a punctured elliptic curve, respectively. The corresponding twisted intersection pairings define their skew-Hermitian forms.

\subsection{The Deligne--Mostow skew-Hermitian space}
Let $\mu=(\mu_1,\ldots,\mu_{n+3})$ be a Deligne--Mostow weight system, where $\mu_j \in (0,1)\cap\QQ$ and $\sum_j\mu_j=2$. Let $A=\{x_1,\dots,x_{n+3}\}\subset \PP^1$ and let $\mathcal{L}_\mu$ be the associated rank one unitary local system. The locally finite twisted homology group
\[
H_1^{\mathrm{lf}}(\PP^1 - A,\mathcal{L}_\mu^{\vee})
\]
admits a natural skew-Hermitian form coming from twisted intersection. It is equivalent to the skew-Hermitian form on $H^1(\PP^1 - A,\mathcal{L}_\mu)$ via regularization and Poincar\'e duality (see \cite[2.18]{deligne1986monodromy}).
Choose arcs joining adjacent punctures and label the corresponding regularized twisted cycles by $\delta_j$, where $\delta_j$ is supported on an arc from $x_j$ to $x_{j+1}$, $j=1,\ldots,n+1$. Set $\lambda_j=e^{2\pi i\mu_j}$. 
In the basis $(\delta_1,\ldots,\delta_{n+1})$, the skew-Hermitian form $\Psi_{\mathrm{DM}}(\mu)$ is represented by an explicit tridiagonal matrix (see \cite[Proposition 7.4]{yu2024comm}).

\subsubsection*{Collision of weights.}
Suppose that the last two punctures $x_{n+2}$ and $x_{n+3}$ with weights $\mu_{n+2}$ and $\mu_{n+3}$ are merged, producing the reduced weight vector
\[
\mu'=(\mu_1,\ldots,\mu_{n+1},\mu_{n+2}+\mu_{n+3}).
\]
Assume that $\mu_{n+2} + \mu_{n+3} < 1$. Denote by $A'$ the set of punctures corresponding to $\mu'$. 
There is an isomorphism of skew-Hermitian spaces
\[
(H_1^{\mathrm{lf}}(\PP^1 - A,\mathcal{L}_\mu^{\vee}),\Psi_{\mathrm{DM}}(\mu))\cong (H_1^{\mathrm{lf}}(\PP^1 - A',\mathcal{L}_{\mu'}^{\vee}),\Psi_{\mathrm{DM}}(\mu'))\oplus \langle\gamma_{n+2,n+3}\rangle,
\]
with the one-dimensional summand satisfying $\Psi_{\mathrm{DM}}(\gamma_{n+2,n+3},\gamma_{n+2,n+3}) = \frac{1-\lambda_{n+2}\lambda_{n+3}}{(1-\lambda_{n+2})(1-\lambda_{n+3})}$ (see \cite[Proposition 7.5]{yu2024comm}).
Thus we have the reduction formula
\[
\det\Psi_{\mathrm{DM}}(\mu) \equiv \det\Psi_{\mathrm{DM}}(\mu')\cdot\frac{1-\lambda_{n+2}\lambda_{n+3}}{(1-\lambda_{n+2})(1-\lambda_{n+3})}.
\]
We later establish the corresponding formula in genus one (see Proposition \ref{prop:elliptic-merge-punctures}).

\subsection{The Ghazouani--Pirio skew-Hermitian space}
We recall the rank-one local system underlying the genus one construction of \cite[\S3.2]{ghazouani2017flat}.
Fix real numbers $\alpha_1,\dots,\alpha_r\in (-1,\infty)- \ZZ_{\ge 0}$ with
\[
\sum_{j=1}^r \alpha_j=0.
\]
Let $\tau\in \HH$ and let $E_\tau=\CC/(\ZZ\oplus \ZZ\tau)$ be the associated elliptic curve. For $z=(z_1,\dots,z_r)\in \CC^r$ with $z_1=0$ and $[z_1],\dots,[z_r]$ pairwise distinct in $E_\tau$, we denote by
\[
E_{\tau,z}:=E_\tau - \{[z_1],\dots,[z_r]\}.
\]

\subsubsection*{Standard generators and monodromy character.}
Let $\beta_0,\beta_\infty,\beta_1,\dots,\beta_r$ be the standard generators of $H_1(E_{\tau,z},\ZZ)$ used in \cite[\S3.2]{ghazouani2017flat}: the cycles $\beta_j$ for $1\leq j\leq r$ are small positively oriented circles around the punctures, and $\beta_0$ and $\beta_\infty$ are the horizontal and vertical generators of the torus. Following \cite[Lemma 3.2.2 and \S4.2.3]{ghazouani2017flat}, we denote by
\[
\alpha_0:=-\frac{\mathrm{Im}\bigl(\sum_{j=1}^r \alpha_j z_j\bigr)}{\mathrm{Im}(\tau)}, \quad \alpha_\infty:=\alpha_0\tau+\sum_{j=1}^r \alpha_j z_j.
\]
Denote by $\rho_j:=\exp(2\pi i \alpha_j)$ for $j\in\{0,1,\dots,r,\infty\}$. Since $\prod_{j=1}^r\rho_j=\exp(2\pi i\sum_{j=1}^r\alpha_j)=1$, we can define a character
\[
\rho\colon H_1(E_{\tau,z},\ZZ)\longrightarrow \CC^\times, \quad \beta_j\mapsto \rho_j.
\]
By definition, $\alpha_0\in\RR$ and $\operatorname{Im}(\alpha_\infty)=\alpha_0\operatorname{Im}(\tau)+\operatorname{Im}(\sum_{j=1}^r\alpha_jz_j)=0$. Thus $\rho$ is unitary. Let $L_{\tau,z}^{\vee}$ be the associated rank-one local system and $L_{\tau,z}$ be its dual.

\subsubsection*{Regularized twisted cycles.}
We use the "nice position" picture from \cite[\S3.2.4]{ghazouani2017flat}. Inside a fundamental parallelogram, we choose oriented arcs from $z_1=0$ to $1$, to $\tau$, and to $z_2,\dots,z_r$ pairwise disjoint away from $z_1$ (see Figure \ref{figure: nice position and basis}).
We define the following locally finite twisted chains
\[
l_0,l_\infty,l_2,\dots,l_r\in H_1^{\mathrm{lf}}(E_{\tau,z},L_{\tau,z}^{\vee})
\]
by tensoring these arcs with compatible local coefficients in $L_{\tau,z}^{\vee}$.
Their regularizations are denoted by $\gamma_0,\gamma_\infty,\gamma_2,\dots,\gamma_r\in H_1(E_{\tau,z},L_{\tau,z}^{\vee})$.
The next proposition is the genus one analogue of the genus zero regularization (\cite[Proposition 2.6.1]{deligne1986monodromy}).

\begin{prop}[{\cite[Proposition 3.3.1]{ghazouani2017flat}}]
\label{prop:elliptic-regularized-basis}
Assume that the punctures are in nice position and the monodromy character is unitary.
\begin{enumerate}[(1)]
\item The regularization map induces an isomorphism
\[
\reg\colon H_1^{\mathrm{lf}}(E_{\tau,z},L_{\tau,z}^{\vee})\xrightarrow{\ \sim\ } H_1(E_{\tau,z},L_{\tau,z}^{\vee}).
\]
\item The cycles $\gamma_\infty,\gamma_0,\gamma_2,\dots,\gamma_r$ satisfy the unique linear relation
\[
(1-\rho_\infty)\gamma_0+(\rho_0-1)\gamma_\infty+\sum_{k=2}^r \frac{\rho_k-1}{\rho_1\cdots\rho_k}\gamma_k=0.
\]
\item The twisted homology group has dimension $r$, and $\gamma_\infty,\gamma_0,\gamma_3,\dots,\gamma_r$ form a basis of $H_1(E_{\tau,z},L_{\tau,z}^{\vee})$.
\end{enumerate}
\end{prop}

\subsubsection*{Skew-Hermitian pairing.} 
Since $\rho$ is unitary, the complex conjugation identifies $L_{\tau,z}^{\vee}$ with $\overline{L}_{\tau,z}$. The twisted intersection product for rank-one local systems on punctured elliptic curves induces a nondegenerate skew-Hermitian form (see \cite[\S 3.1.3--3.1.6]{ghazouani2017flat})
\[
\Psi_\rho\colon
H_1(E_{\tau,z},L_{\tau,z}^{\vee})\times
H_1(E_{\tau,z},L_{\tau,z}^{\vee})
\longrightarrow \CC,
\quad
\Psi_\rho(\gamma,\delta)=\gamma\cdot\overline{\delta}.
\]

\begin{figure}[htp]
  \centering
  \includegraphics[width=10cm]{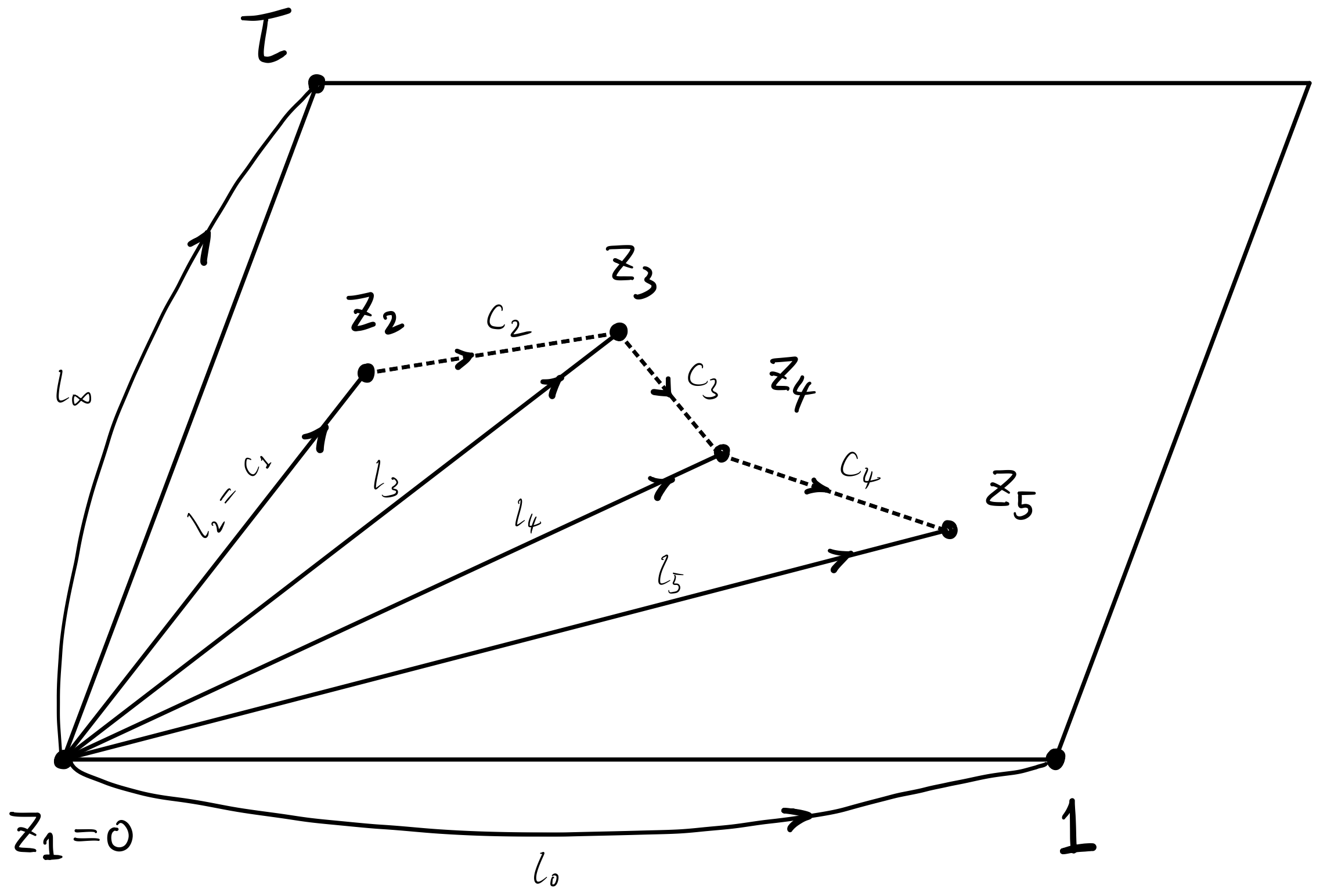}
  \caption{The oriented arcs underlying the locally finite twisted chains $l_j$ for $j=0,\infty,2,\ldots,r$ and $c_j$ for $j=1,\ldots,r-1$, with the points $z_i$ in very nice position.}
  \label{figure: nice position and basis}
\end{figure}

\subsubsection*{The twisted intersection matrix.}
We choose oriented arcs joining $z_j$ to $z_{j+1}$, for $1\leq j\leq r-1$ (see Figure \ref{figure: nice position and basis} and \cite[\S 3.2.4, Figure 5]{ghazouani2017flat}). Similarly, we define the following locally finite twisted chains 
\[
c_1,c_2,\ldots,c_{r-1}
\]
by tensoring these arcs with compatible local sections of $L_{\tau,z}^{\vee}$.
Choose the local coefficients such that $c_1=l_2$. 
The linear transformation between $ \gamma_2,\ldots,\gamma_r $ and $ \delta_1,\ldots,\delta_{r-1} $ is given by
\[
\gamma_i = \sum_{j=1}^{i-1} \delta_j, \quad 2 \le i \le r.
\]
Then $l_0,l_\infty,c_2,\ldots,c_{r-1}$ form a basis of $H_1^{\mathrm{lf}}(E_{\tau,z},L_{\tau,z}^{\vee})$.
The regularizations of $c_1=l_2,c_2,\ldots,c_{r-1}$ are denoted by $\delta_1 = \gamma_2,\delta_2,\ldots,\delta_{r-1} \in H_1(E_{\tau,z},L_{\tau,z}^{\vee})$.

In the case of $r=2$, the relation between the regularized cycles is (see \cite[\S 3.5]{ghazouani2017flat})
\[
(1-\rho_\infty)\gamma_0-(1-\rho_0)\gamma_\infty=(1-\rho_2)\gamma_2,
\]
so the twisted homology is generated by $\gamma_\infty$ and $\gamma_0$. 

\begin{prop}[{\cite[\S 3.5]{ghazouani2017flat}}]
\label{prop:r2-elliptic-matrix}
In the case $r=2$, with respect to the basis $(\gamma_\infty,\gamma_0)$ and the dual locally finite basis $(\overline{l}_\infty,\overline{l}_0)$, the twisted intersection matrix is
\[
I\!I_\rho :=
\left(
\begin{array}{cc}
\dfrac{(\rho_\infty-1) (\rho_1\rho_{\infty}-1)}{(\rho_1-1)\rho_\infty}
&
\dfrac{-1+\rho_0^{-1}+\rho_{\infty}-\rho_1\rho_{\infty}\rho_0^{-1}}{\rho_1-1} \\
\dfrac{-\rho_1+\rho_0\rho_1+\rho_1\rho_{\infty}^{-1}-\rho_0\rho_{\infty}^{-1}}{\rho_1-1}
&
\dfrac{\rho_0-1}{\rho_1-1}\bigl(1-\dfrac{\rho_1}{\rho_0}\bigr)
\end{array}
\right).
\]
This is the $2\times 2$ matrix of the skew-Hermitian form $\Psi_\rho$. Moreover, one has $\det I\!I_\rho=1$.
\end{prop}

\subsection{The arithmetic commensurability criterion}
\label{subsec:comm}

Let $K$ be a CM field and $F=K\cap\RR$ be its totally real subfield. If $\Psi$ is a $K$-skew-Hermitian form, then for any $\zeta\in K-F$ the form $(\zeta-\overline{\zeta})\Psi$ is $K$-Hermitian. Landherr \cite{landherr1935aquivalenz} classifies $K$-Hermitian forms over the CM extension $K/F$ by their signatures at the real embeddings of $F$ and their determinant class in $F^\times/N_{K/F}(K^\times)$ (see also \cite[p.268, Example 5]{Jacobson1940anoteonhermitianforms}). Yu and Zheng \cite[\S6]{yu2024comm} discuss the relation between $F$-conformal Hermitian spaces, the induced $F$-forms of projective unitary groups, and commensurability of arithmetic subgroups.
The results in \cite[\S6]{yu2024comm} on conformal $K$-skew-Hermitian spaces and the associated $F$-forms give the following criterion.

\begin{prop}
\label{prop:parity-arithmetic-criterion}
Let $K_1, K_2$ be two imaginary quadratic extensions of $\QQ$. 
Let $(V_i, \Psi_i)$ be a $K_i$-skew-Hermitian space of dimension $n+1 \ge 3$.
Let $\Gamma_i$ be an arithmetic subgroup of the $\QQ$-group
$\PU(V_i,\Psi_i)$, viewed as a lattice in $\PU(1,n)$.
\begin{enumerate}[(1)]
\item If $n$ is even, then $\Gamma_1$ and $\Gamma_2$ are commensurable if and only if $K_1=K_2$.
\item If $n\ge 3$ is odd, then $\Gamma_1$ and $\Gamma_2$ are commensurable if and only if $K_1=K_2=K$ and
\[
\det \Psi_1=\det \Psi_2 \in \QQ^\times/N_{K/\QQ}(K^\times).
\]
\end{enumerate}
\end{prop}

\begin{proof}
Let $G_i=\PU(V_i,\Psi_i)$. Fix $\eta_i\in K_i^\times$ with $\overline{\eta_i}=-\eta_i$ such that $h_i:=\eta_i\Psi_i$ has signature $(1,n)$. Then $h_i$ is Hermitian and $\PU(V_i,h_i)=G_i$.

Suppose that $\Gamma_1$ and $\Gamma_2$ are commensurable. After conjugating $\Gamma_2$, we may assume that $\Gamma_1\cap\Gamma_2$ has finite index in both groups. By \cite[Proposition 6.8]{yu2024comm}, it follows that $G_1$ is isomorphic to $G_2$ as $\QQ$-algebraic groups. Hence \cite[Proposition 6.4(1)]{yu2024comm} gives $K_1=K_2=K$ and $\QQ$-conformality of $h_1$ and $h_2$. Conversely, $\QQ$-conformality identifies $G_1$ and $G_2$ over $\QQ$. Under this identification, $\Gamma_1$ and $\Gamma_2$ are arithmetic subgroups of the same $\QQ$-group, hence commensurable.

By \cite[Proposition 6.2]{yu2024comm}, $h_1$ and $h_2$ are $\QQ$-conformal when $n$ is even, whereas for odd $n$ they are $\QQ$-conformal if and only if their determinant classes are equal.
In this case, $n+1$ is even and $\eta_1/\eta_2\in\QQ^\times$, with $(\eta_1/\eta_2)^{n+1} = N_{K/\QQ}((\eta_1/\eta_2)^{(n+1)/2})$. Thus $\det h_1=\det h_2$ if and only if $\det\Psi_1=\det\Psi_2$.
\end{proof}

The CM fields $\QQ(i)$ and $\QQ(\zeta_3)$ occur in the comparison of determinant classes of the Ghazouani--Pirio and Deligne--Mostow--Thurston forms in Section \ref{sec:comparison}.

\section{Genus One Degeneration and Determinant Classes}
\label{subsec:degeneration}

The collision of two punctures produces an orthogonal decomposition of $\Psi_\rho$ into a lower-dimensional skew-Hermitian form and a one-dimensional summand generated by the vanishing cycle. The determinant calculation is reduced to the two-puncture case by itarations.
This method was established in \cite[\S 7]{yu2024comm} for the genus zero case.

\subsection{Collision of punctures and reduction to the two-puncture case}

From now on, assume that the monodromy character $\rho$ has finite order at least $3$. Denote by
\[
K:=\mathbb{Q}(\operatorname{Im}\rho),
\quad
F:=K\cap\mathbb{R}.
\]
Then $K/F$ is a CM extension.
The skew-Hermitian form $\Psi_\rho$ is defined over $K$, and the determinant class of $\Psi_\rho$, denoted by $\det \Psi_\rho$, is an element in $K^\times/N_{K/F}(K^\times)$.
 
The underlying oriented arcs of $l_0$ and $l_\infty$ represent a basis of $H_1(E_\tau,\mathbb Z)$ (see Figure \ref{figure: nice position and basis}).
Let $z_{r-1}$ and $z_r$ collide inside a simply connected disk which is disjoint from these two curves and contains no other punctures.
Let $\rho'$ be the new monodromy datum obtained by replacing $(\rho_{r-1},\rho_r)$ with the single merged value $\rho_{r-1}\rho_r$. Assume that the regularized vanishing cycle $\delta_{r-1}$ is supported in the shrinking disk.

\begin{prop}
\label{prop:elliptic-merge-punctures}
Assume that $r\ge 3$ and $\rho_{r-1}\rho_r\neq 1$. Then there is an orthogonal decomposition of $K$-skew-Hermitian forms
\[
\Psi_\rho\cong \Psi_{\rho'}\oplus \langle \delta_{r-1}\rangle.
\]
Here $\Psi_\rho(\delta_{r-1},\delta_{r-1})
=-1+\frac{1}{1-\rho_{r-1}}+\frac{1}{1-\rho_r}
=
\frac{1-\rho_{r-1}\rho_r}{(1-\rho_{r-1})(1-\rho_r)}$.
\end{prop}

\begin{proof}
The calculation is local and hence the same as in \cite[Propositions 7.4--7.5]{yu2024comm}.
\end{proof}

\subsection{The determinant formula}

\begin{prop}
\label{prop:elliptic-determinant-reduction}
Let $\rho_1,\dots,\rho_r$ be the puncture monodromies, and assume that they can be successively merged until only two remain, with the monodromy produced at each merging step different from $1$. Then
\[
\det \Psi_\rho
\equiv
\frac{1}{\prod_{j=1}^r (1-\rho_j)}
\quad
\text{in }K^\times/N_{K/F}(K^\times).
\]
In particular, for fixed $K/F$, the determinant class in $K^\times/N_{K/F}(K^\times)$ is determined by the puncture monodromies $\rho_1,\dots,\rho_r$.
\end{prop}

\begin{proof}
Choose a sequence of merging steps as in the statement, and let $\rho^{(m)}$ denote the datum with $m$ puncture monodromies in this sequence. Thus $\rho^{(r)}=\rho$ and $\rho^{(2)}=(\rho_0,\rho_\infty,t,t^{-1})$. For $3\leq m\leq r$, let $(\nu_{m,1},\nu_{m,2})$ be the pair merged in passing from $\rho^{(m)}$ to $\rho^{(m-1)}$.
By Proposition \ref{prop:elliptic-merge-punctures}, we obtain the formula
\[
\det \Psi_{\rho^{(m)}} \equiv \det \Psi_{\rho^{(m-1)}}\cdot \frac{1-\nu_{m,1}\nu_{m,2}}{(1-\nu_{m,1})(1-\nu_{m,2})}.
\]
Denote by $\Sigma_m$ the multiset of puncture monodromies for $\rho^{(m)}$ (monodromies are counted with multiplicity). Then
\[
\frac{1-\nu_{m,1}\nu_{m,2}}{(1-\nu_{m,1})(1-\nu_{m,2})}
=
\frac{\prod_{\xi\in \Sigma_{m-1}}(1-\xi)}{\prod_{\eta\in \Sigma_m}(1-\eta)}.
\]
Since the reduced two-puncture form has determinant $1$ by Proposition \ref{prop:r2-elliptic-matrix}, iterating gives
\[
\det \Psi_\rho \equiv
\frac{\prod_{\xi\in \Sigma_2}(1-\xi)}{\prod_{\eta\in \Sigma_r}(1-\eta)}
=
\frac{(1-t)(1-t^{-1})}{\prod_{j=1}^r(1-\rho_j)}.
\]
Finally, notice that $(1-t)(1-t^{-1})=(1-t)(1-\overline{t})=N_{K/F}(1-t)$, so this factor becomes trivial in $K^\times/N_{K/F}(K^\times)$ and the simplified formula follows.
\end{proof}

\section{Comparison with Arithmetic Deligne--Mostow--Thurston Lattices}
\label{sec:comparison}

We apply the determinant formula to the sixteen arithmetic Ghazouani--Pirio examples and compare them with the arithmetic commensurability table in \cite[\S 9.2, Table 2]{yu2024comm}.

\subsection{The arithmetic Ghazouani--Pirio list}
The Ghazouani--Pirio list \cite[\S 11.3, Table 1]{ghazouani2017moduli} has sixteen labels $\ell\in\{a,b,\dots,p\}$ corresponding to the following data
\[
m(\ell)\in\{3,4,6\},
\quad
\frac{m(\ell)\theta(\ell)}{2\pi}=\big(a_1(\ell),\dots,a_r(\ell)\big)\in \ZZ^r.
\]
For these arithmetic leaves, the puncture monodromies are $\rho_j(\ell)=\exp(i\theta_j(\ell))=\zeta_{m(\ell)}^{a_j(\ell)}$. By Proposition \ref{prop:elliptic-determinant-reduction}, we have
\[
\det \Psi_\ell
\equiv
\frac{1}{\prod_{j=1}^{r(\ell)} \big(1-\zeta_{m(\ell)}^{a_j(\ell)}\big)}
\in \QQ(\zeta_{m(\ell)})^\times/N_{\QQ(\zeta_{m(\ell)})/(\QQ(\zeta_{m(\ell)})\cap \RR)}\big(\QQ(\zeta_{m(\ell)})^\times\big).
\]

\begin{prop}
\label{prop:gafa-arithmetic-determinants}
For the labels $\ell$ in the Ghazouani--Pirio list, the determinant classes of $\Psi_\ell$ have the following representatives:
\begin{center}
\begin{tabular}{ccl}
\hline
$\ell$ & $m(\ell)$ & $\det\Psi_\ell$ \\
\hline
$a,i$   & $3$ & $\frac{1}{(1-\zeta_3^2)^3}$ \\
$b,d$   & $4$ & $\frac{1}{(1-\zeta_4^2)(1-\zeta_4^3)^2}$ \\
$j$     & $4$ & $\frac{1}{(1-\zeta_4^3)^4}$ \\
$c,g$   & $6$ & $\frac{1}{(1-\zeta_6^2)(1-\zeta_6^5)^2}$ \\
$e,f,h$ & $6$ & $\frac{1}{(1-\zeta_6^3)(1-\zeta_6^4)(1-\zeta_6^5)}$ \\
$k$     & $6$ & $\frac{1}{(1-\zeta_6^3)(1-\zeta_6^5)^3}$ \\
$l,m$   & $6$ & $\frac{1}{(1-\zeta_6^4)^2(1-\zeta_6^5)^2}$ \\
$n,o$   & $6$ & $\frac{1}{(1-\zeta_6^4)(1-\zeta_6^5)^4}$ \\
$p$     & $6$ & $\frac{1}{(1-\zeta_6^5)^6}$ \\
\hline
\end{tabular}
\end{center}
\end{prop}

\begin{proof}
A direct computation shows that each datum in the Ghazouani--Pirio list can be successively merged until only two remain, with the monodromy produced at each merging step different from $1$.
Applying the formula in Proposition \ref{prop:elliptic-determinant-reduction}, we obtain the listed expressions. 
\end{proof}

\subsection{Determinant classes for Deligne--Mostow--Thurston tuples}

Recall that a Ghazouani--Pirio datum with $r$ punctures has ball dimension $r-1$ and a Deligne--Mostow--Thurston datum with $N$ points has ball dimension $N-3$. The two ball dimensions coincide when $r = N-2$.

The following lemma follows from the Hasse--Minkowski theorem. See, for instance, \cite[Chapter III, \S1--2, and Chapter IV, \S3.2]{serre1973course}.

\begin{lem}
\label{lem:rational-norm-criterion}
Let $q\in\QQ^\times$.
\begin{enumerate}[(1)]
\item The number $q$ lies in $N_{\QQ(i)/\QQ}(\QQ(i)^\times)$ if and only if $q=a^2+b^2$ for some $a,b\in\QQ$. Equivalently, $q>0$ and $v_p(q)$ is even for every prime $p\equiv 3\pmod 4$.
\item Note that $\QQ(\zeta_3) = \QQ(\sqrt{-3})$. The number $q$ lies in $N_{\QQ(\zeta_3)/\QQ}(\QQ(\zeta_3)^\times)$ if and only if $q=a^2+3b^2$ for some $a,b\in\QQ$. Equivalently, $q>0$ and $v_p(q)$ is even for every prime $p\equiv 2\pmod 3$.
\end{enumerate}
\end{lem}

\begin{prop}
\label{prop:gafa-dm-comparison}
For each pair $(\ell,\mu)$ below, set
$q(\ell,\mu):=\det\Psi_\ell/\det\Psi_{\mathrm{DM}}(\mu)$.
In the odd ball dimensions, the determinant comparison with the arithmetic Deligne--Mostow--Thurston tuples over the same field is as follows.

\begin{enumerate}[(1)]
\item In dimension $3$ and field $\QQ(i)$, $\det\Psi_j$ has the same class as the determinants of $\frac{1}{4}(1,1,1,1,1,3)$ and $\frac{1}{4}(1,1,1,1,2,2)$.

\noindent
The corresponding values of $q(j,\mu)$, in the same order, are $2$ and
$4$.

\item In dimension $3$ and field $\QQ(\zeta_3)$, $\det\Psi_l=\det\Psi_m$ has the same class as the determinants of
$\frac{1}{3}(1,1,1,1,1,1)$,
$\frac{1}{6}(1,1,1,1,4,4)$,
$\frac{1}{6}(1,1,1,2,2,5)$,
$\frac{1}{6}(1,1,2,2,2,4)$ and
$\frac{1}{6}(1,1,2,2,3,3)$.

\noindent
The corresponding values of $q(l,\mu)=q(m,\mu)$ are $9,1,1,3,4$. 

\noindent
The determinant $\det\Psi_k$ has the same class as the determinants of
$\frac{1}{6}(1,1,1,1,3,5)$,
$\frac{1}{6}(1,1,1,2,3,4)$,
$\frac{1}{6}(1,1,1,3,3,3)$ and
$\frac{1}{6}(1,2,2,2,2,3)$.

\noindent
The corresponding values of $q(k,\mu)$ are $1,3,4,9$.

\item In dimension $5$ and field $\QQ(\zeta_3)$, $\det\Psi_p$ has the same class as the determinants of
$\frac{1}{6}(1,1,1,1,1,1,1,5)$,
$\frac{1}{6}(1,1,1,1,1,1,2,4)$,
$\frac{1}{6}(1,1,1,1,1,1,3,3)$ and
$\frac{1}{6}(1,1,1,1,2,2,2,2)$.

\noindent
The corresponding values of $q(p,\mu)$ are $1,3,4,9$. 
\end{enumerate}
\end{prop}

\begin{proof}
For a Ghazouani--Pirio label $\ell$ and a Deligne--Mostow--Thurston tuple $\mu=\frac{1}{d}(b_1,\ldots,b_N)$, Propositions \ref{prop:elliptic-determinant-reduction} and \ref{prop:gafa-arithmetic-determinants}, together with \cite[\S 7.3]{yu2024comm}, give
\[
q(\ell,\mu) \equiv \frac{\prod_{j=1}^N(1-\zeta_d^{b_j})}{\prod_{j=1}^r(1-\rho_j(\ell))}.
\]
This formula gives the following values of $q(\ell,\mu)$, listed in the same order as the corresponding Deligne--Mostow--Thurston tuples in the statement:
\begin{center}
\begin{tabular}{cc}
\hline
$\ell$ & $q(\ell,\mu)$ \\
\hline
$j$   & $2,4$ \\
$l,m$ & $9,1,1,3,4$ \\
$k$   & $1,3,4,9$ \\
$p$   & $1,3,4,9$ \\
\hline
\end{tabular}
\end{center}
Now $2=N_{\QQ(i)/\QQ}(1+i)$ and $4=N_{\QQ(i)/\QQ}(2)$, while $1,3,4,9$ are norms from $\QQ(\zeta_3)$. Hence these quotients are trivial by Lemma \ref{lem:rational-norm-criterion}. 
\end{proof}

\subsection{Proof of the main theorem}

Let $\Gamma_\mu$ ($\Gamma_\ell$, respectively) denote the monodromy group for a Deligne--Mostow--Thurston tuple $\mu$ (a Ghazouani--Pirio label $\ell$, respectively). The twisted intersection forms $\Psi_{\mathrm{DM}}(\mu)$ and $\Psi_\ell$ are invariant under the corresponding monodromy actions and 
\[
\Gamma_\ell\subset\PU(\Psi_\ell),
\quad
\Gamma_\mu\subset\PU(\Psi_{\mathrm{DM}}(\mu)).
\]
For the labels and tuples considered in this paper, these are arithmetic subgroups of the respective $\QQ$-groups by \cite[\S 11.3]{ghazouani2017moduli} and \cite[\S 9.2]{yu2024comm}. 

\begin{proof}[Proof of Theorem \ref{thm:main}]
We first show that all groups appearing in the same row of Table \ref{tab:arithmetic-gp-dm-classification} are commensurable.
Suppose first that the ball dimension is even. The coincidence of the CM fields implies commensurability by Proposition \ref{prop:parity-arithmetic-criterion} (1). This proves the cases of dimension $2$ and $4$ in Table \ref{tab:arithmetic-gp-dm-classification}.

Now suppose that the ball dimension is odd. For every pair $(\ell,\mu)$ appearing in the same row of dimension $3$ or $5$, $\Psi_\ell$ and $\Psi_{\mathrm{DM}}(\mu)$ have the same determinant class by Proposition \ref{prop:gafa-dm-comparison}. They also have the same dimension. It follows from Proposition \ref{prop:parity-arithmetic-criterion} (2) that the corresponding monodromy groups are commensurable.

The two rows of dimension $3$ over $\QQ(\zeta_3)$ are distinct. Indeed, Proposition \ref{prop:gafa-arithmetic-determinants} gives
\[
\frac{\det\Psi_l}{\det\Psi_k}
=
\frac{\det\Psi_m}{\det\Psi_k}
=
\frac23.
\]
Since $v_2(2/3)=1$ and $2\equiv2\pmod 3$, this quotient is not a norm from $\QQ(\zeta_3)$ by Lemma \ref{lem:rational-norm-criterion}. Hence no group in one of these rows is commensurable with a group in the other.

For each $\ell$, we find a Deligne--Mostow--Thurston group commensurable with $\Gamma_\ell$. The commensurability classification in \cite[\S 9.2]{yu2024comm} gives exactly the tuples listed in the corresponding row of Table \ref{tab:arithmetic-gp-dm-classification}.
\end{proof}

\bibliography{ref}
\bibliographystyle{alpha}
\Addresses

\end{document}